\documentclass[a4paper,10pt]{article}

\usepackage[utf8]{inputenc} % Linux
\usepackage{amsmath}
\usepackage{amsbsy}
\usepackage{amssymb}
\usepackage{amscd,amsthm}
\usepackage{graphicx}
\usepackage[mathcal]{eucal}
\usepackage{verbatim}
\usepackage[english]{babel}
\usepackage[dvigs]{epsfig}
\usepackage{dsfont}
\usepackage{longtable}
\usepackage{marvosym}
\usepackage{anysize}

\newcommand{\tl}{\text{li}}

\newcommand{\N}{\mathbb{N}}
\newcommand{\R}{\mathbb{R}}

\newtheorem{thm}{Theorem}[section]

\newtheorem{kor}[thm]{Corollary}
\newtheorem{prop}[thm]{Proposition}

\theoremstyle{remark}
\newtheorem*{rema}{Remark}

\theoremstyle{definition}
\newtheorem*{defi}{Definition}

\title{An explicit upper bound for the first $k$-Ramanujan prime}
\author{Christian Axler \\ Mathematical Institut \\ Heinrich-Heine-University, Düsseldorf, Germany \\ axler@math.uni-duesseldorf.de \vspace{0.5cm} \\ Thomas
Leßmann \\ Mathematical Institut \\ Heinrich-Heine-University, Düsseldorf, Germany \\ lessmann@math.uni-duesseldorf.de}

\begin{document}

\maketitle

\begin{abstract}
In this paper we establish an explicit upper bound for the first $k$-Ramanujan prime $R_1^{(k)}$ by using a recent result concerning the existence of
prime numbers in small intervals.
\end{abstract}

\section{Introduction}

% Ramanujan primes, named for the Indian mathematician Srinivasa Ramanujan, were introduced by Sondow \cite{so3} in 2005 and have their origin in Bertrand's
% postulate.
% 
% \begin{bp}
% For each $n \in \mathds{N}$ there is a prime number $p$ with $n < p \leq 2n$.
% \end{bp}
% 
% \noindent
% Bertrand's postulate was proved, for instance, by Tchebychev \cite{tch} and by Erdös \cite{er2}. In 1919, Ramanujan \cite{ram} proved an extension of
% Bertrand's postulate by showing that
% \begin{displaymath}
% \pi(x) - \pi \left( \frac{x}{2} \right) \geq 1 \; (\text{respectively} \; 2,3,4,5,\ldots)
% \end{displaymath}
% for every
% \begin{displaymath}
% x \geq 2 \; ( \text{respectively} \; 11,17,29,41, \ldots).
% \end{displaymath}
% Motivated by the fact $\pi(x) - \pi(x/2) \rightarrow \infty$ as $x \rightarrow \infty$ by the Prime Number Theorem (PNT), Sondow \cite{so3} defined the
% number $R_n \in \N$ for each $n \in \N$ as the smallest positive integer such that the inequality $\pi(x) - \pi(x/2) \geq n$ holds for every $x \geq R_n$.
% He called the number $R_n$ the \emph{$n$th Ramanujan prime}, because $R_n \in \Pz$ for every $n\in \N$, where $\Pz$ denotes the set of prime numbers.
% 
% This can be generalized as follows. Let $k \in (1,\infty)$. Again, t
Let $k \in (1,\infty)$. The PNT implies that $\pi(x) - \pi(x/k) \to \infty$ as $x \to \infty$ and Shevelev \cite{sh} introduced the $n$th $k$-Ramanujan
prime as follows.

\begin{defi}
Let $k>1$ be real. For every $n \in \N$, let
\begin{displaymath}
R_n^{(k)} = \min \{ m \in \N \mid \pi(x) - \pi(x/k) \geq n \;\, \text{for every real} \; x \geq m \}.
\end{displaymath}
It is easy to show that this number is prime and it is called the \emph{$n$}th \emph{$k$-Ramanujan prime}.
\end{defi}

In this paper we give an explicit upper bound for the first $k$-Ramanujan prime $R_1^{(k)}$ for small $k$. In order to do this, we first give some
known results on the existence of prime numbers in short intervals.

\section{On the existence of prime numbers in short intervals}

Bertrand's postulate states that for every $n \in \N$ there is always a prime in the interval $(n, 2n]$. Now, we note some improvements of this result. In
2003, Ramar\'{e} \& Saouter \cite{rasa} showed that for every $x \geq 10726905041$ the interval
\begin{displaymath}
(x, x + x/28313999]
\end{displaymath}
always contains a prime number. This was improved by Dusart \cite{pd4} in 2010 by showing that for every $x \geq 396738$ there is always a prime number $p$
with
\begin{equation} \label{gl146}
x < p \leq x \left( 1 + \frac{1}{25 \log^2 x} \right).
\end{equation}
In 2014, Trudgian \cite{trud} proved that for every $x \geq 2898239$ there exists a prime number $p$ such that
\begin{displaymath}
x < p \leq x \left( 1 + \frac{1}{111 \log^2 x} \right).
\end{displaymath}
Recently, in \cite{ax} it is shown that the following result holds.

\begin{prop} \label{p201}
For every $x \geq 58837$ there is a prime number $p$ such that
\begin{displaymath}
x < p \leq x \left( 1 + \frac{1.188}{\log^3 x} \right).
\end{displaymath}
\end{prop}

\section{On an upper bound for the first $k$-Ramanujan prime}

Let $n \in \N, c>0$ and $x_0 > 0$ so that for every $x \geq x_0$ there is a prime $p$ such that
\begin{equation} \label{gl2}
x < p \leq x \left( 1 + \frac{c}{\log^n x} \right).
\end{equation}
% ; i.e. for every $x \geq x_0$ we have
% \begin{equation} \label{gl2}
% \pi \left( x\left(1+\frac{c}{\log^n x} \right) \right) - \pi(x) \geq 1.
% \end{equation}
Then, we obtain the following result.

\begin{prop} \label{t301}
Let $x \geq x_0$ and $k = 1+c/\log^nx$. Then
\begin{displaymath}
R_1^{(k)} \leq kx. % \left( 1 + \frac{c}{\log^n x} \right).
\end{displaymath}
\end{prop}

\begin{proof}
Let $y \geq kx$. From \eqref{gl2} we obtain the existence of a prime $p$ in
\begin{displaymath}
\left( \frac{y}{k}, \frac{y}{k}\left( 1 + \frac{c}{\log^n(y/k)} \right) \right].
\end{displaymath}
Since $y/k \geq x$, we get
\begin{displaymath}
k \geq 1 + \frac{c}{\log^n(y/k)},
\end{displaymath}
so that $p \in (y/k,y]$.
% To show that there is a prime in $(y/k, y]$, we apply \eqref{gl2} to obtain  
% Let $x \geq x_0$ and let $y \geq x(1+c/\log^nx)$. Define $y_0 \in \R$ so that
% \begin{displaymath}
% y = y_0 \left(1+\frac{c}{\log^n y_0} \right).
% \end{displaymath}
% Then $y_0 \geq x \geq x_0$. Since
% \begin{equation} \label{gl3}
% \frac{c}{\log^n x} \geq \frac{c}{\log^n y_0},
% \end{equation}
% it follows that the inequality
% \begin{align*}
% \pi(y) - \pi \left( \frac{y}{ 1+\frac{c}{\log^n x}} \right) & = \pi \left( y_0\left(1+\frac{c}{\log^n y_0} \right) \right) - \pi \left(
% \frac{y_0\left(1+\frac{c}{\log^n y_0}\right)}{1+\frac{c}{\log^n x}} \right) \\
% & \stackrel{\eqref{gl3}}{\geq } \pi \left( y_0\left(1+\frac{c}{\log^n y_0} \right) \right) - \pi(y_0) \\
% & \stackrel{\eqref{gl2}}{\geq } 1
% \end{align*}
% holds. Now we use the definition of the first $k$-Ramanujan prime.
\end{proof}

\begin{kor} \label{k302}
For every
\begin{displaymath}
k \in \left( 1, 1+ \frac{c}{\log^nx_0} \right],
\end{displaymath}
we have
\begin{displaymath}
R_1^{(k)} \leq k \cdot \exp \left( \sqrt[n]{\frac{c}{k-1}} \right).
\end{displaymath}
\end{kor}

\begin{proof}
Define $x \in \R$ so that
\begin{displaymath}
k = 1+ \frac{c}{\log^nx}.
\end{displaymath}
Then $x \geq x_0$ and by using Proposition \ref{t301} we get
\begin{displaymath}
R_1^{(k)} = R_1^{(1 + c/\log^n x)} \leq x \left( 1 + \frac{c}{\log^n x} \right) = k \cdot \exp \left( \sqrt[n]{\frac{c}{k-1}} \right).
\end{displaymath}
This proves our corollary.
\end{proof}

\section{A characterisation for $k$-Ramanujan primes}

\noindent
We obtain the following useful characterisation for the first $k$-Ramanujan prime.

\begin{prop} \label{t401}
Let $N \in \N$. Then $p_N$ is the first $k$-Ramanujan prime iff the following two conditions are fulfilled:
\begin{enumerate}
 \item[\emph{(a)}] For every $n \geq N$, we have
 \begin{displaymath}
\frac{p_{n+1}}{p_n} \leq k.
 \end{displaymath}
 \item[\emph{(b)}] We have
 \begin{displaymath}
\frac{p_{N}}{p_{N-1}} > k.
 \end{displaymath}
\end{enumerate}
\end{prop}

\begin{proof}
Let $p_N = R_1^{(k)}$. To show (a), we assume that there is an integer $n \geq N$ so that $p_{n+1}/p_n > k$. Let $x = kp_{n}$. Then $p_n < x < p_{n+1}$, so
that 
\begin{equation} \label{401}
\pi(x) - \pi(x/k) = n - n = 0.
\end{equation}
Since $x > p_N = R_1^{(k)}$, the equation \eqref{401} contradicts the definition of $R_1^{(k)}$.
% By the definition of the first $k$-Ramanujan prime, we obtain $\pi(x) - \pi(x/k) \geq 1$. On the other hand, we get
% \begin{displaymath}
% 1 \leq \pi(x) - \pi \left( \frac{x}{k} \right) = \pi(kp_n) - \pi(p_n) \stackrel{\eqref{401}}{=} n-n = 0,
% \end{displaymath}
% which gives a contradiction. 
So, we proved (a). To show (b), we assume that $p_N/p_{N-1} \leq k$. Since $p_N = R_1^{(k)}$, there is a $x_0 \in [p_{N-1},
p_N)$ so that $\pi(x_0) - \pi(x_0/k) = 0$. Since we have $x_0/k < p_N/k \leq p_{N-1}$, we get
\begin{displaymath}
0 = \pi(x_0) - \pi \left( \frac{x_0}{k} \right) > \pi(p_{N-1}) - \pi(p_{N-1}) = 0,
\end{displaymath}
which gives a contradiction.

Now, let (a) and (b) be true. To show that $p_N = R_1^{k}$, we show first that $p_N \geq R_1^{k}$. Let $x \geq p_N$. We assume that $\pi(x) - \pi(x/k) =
0$. Then there exists an integer $n \geq N$ such that $p_n \leq x/k < x < p_{n+1}$. Hence,
\begin{displaymath}
\frac{p_{n+1}}{p_n} > \frac{x}{x/k} = k,
\end{displaymath}
which contadicts (1). Now, we prove that $p_N \leq R_1^{k}$. Let $x = kp_{N-1}$. Then,
\begin{equation} \label{402}
p_{N-1} < x \stackrel{(b)}{<} p_{N}.
\end{equation}
Hence, we obtain
\begin{displaymath}
\pi(x) - \pi \left( \frac{x}{k} \right) \stackrel{\eqref{402}}{=} N-1 - \pi(p_{N-1}) = 0.
\end{displaymath}
It follows that $R_1^{(k)} > x > p_{N-1}$. So $R_1^{(k)} \geq p_N$.
\end{proof}

\section{Numerical results}

In the following proposition we derive an explicit $p$ such that $R_1^{(k)} = p$ for the case $k = 1.0008968291$.

\begin{prop} \label{p501}
We have
\begin{displaymath}
R_1^{(1.0008968291)} = 58889 = p_{5950}.
\end{displaymath}
\end{prop}

\begin{proof}
Let $x_0 = 58837$, $c = 1.188$ and $n=3$. Then
\begin{displaymath}
1.0008968291 \leq 1 + \frac{1.188}{\log^3 58837}.
\end{displaymath}
Using Proposition \ref{p201} and Corollary \ref{k302}, we obtain that the inequality
\begin{displaymath}
R_1^{(1.0008968291)} \leq 1.0008968291 \cdot \exp \left( \sqrt[3]{\frac{1.188}{0.0008968291}} \right) \leq 58890
\end{displaymath}
holds. Since $R_1^{(1.0008968291)}$ is a prime number, we obtain
\begin{displaymath}
R_1^{(1.0008968291)} \leq 58889.
\end{displaymath}
On the other hand we have
\begin{displaymath}
\pi(58888) - \pi \left( \frac{58888}{1.0008968291} \right) = 0,
\end{displaymath}
hence $R_1^{(1.0008968291)} > 58888$.
\end{proof}

\begin{rema}
\begin{enumerate}
 \item[(a)] If $k \geq 5/3$, then $R_1^{(k)} = 2$ (see \cite[Prop. 2.5(ii)]{ax2})
 \item[(b)] If $k \in [1.0008968291, 5/3)$, then, using Proposition \ref{p501}, we obtain
 \begin{displaymath}
 m := \max \{ n \geq 2 \mid p_n/p_{n-1} > k\} = \max \{ n \in \{2, \ldots, 5950\} \mid p_n/p_{n-1} > k\}.
 \end{displaymath}
By Proposition \ref{t401}, it follows $R_1^{(k)} = p_m$.
\end{enumerate}
\end{rema}
\bigskip
\noindent
By using Remark (b) and a computer, we obtain the following

\begin{kor}
\begin{enumerate}
 \item[\emph{(a)}] If
 \begin{displaymath}
 k \in \left[ 1.0008968291, \frac{p_{5950}}{p_{5949}} \right),
 \end{displaymath}
then $R_1^{(k)} = 58889$.
\item[\emph{(b)}] For every $1 \leq n \leq 44$ we define the numbers $a(n)$ by
\setlength{\doublerulesep}{2mm}
\begin{center}
\begin{tabular}{|l|@{\,}|*{11}{@{\ \,}c@{\ \,}|}}
\hline
$n$       & $   1 $ & $    2$ & $    3$ & $    4$ & $    5$ & $    6$ & $    7$ & $    8$ & $    9$ & $   10$ & $   11$ \\ \hline
$a(n)$    & $   3 $ & $    5$ & $    7$ & $   10$ & $   12$ & $   16$ & $   31$ & $   35$ & $   47$ & $   48$ & $   63$ \\ \hline
$p_{a(n)}$& $   5 $ & $   11$ & $   17$ & $   29$ & $   37$ & $   53$ & $  127$ & $  149$ & $  211$ & $  223$ & $  307$ \\ \hline
\hline
$n$       & $   12$ & $   13$ & $   14$ & $   15$ & $   16$ & $   17$ & $   18$ & $   19$ & $   20$ & $   21$ & $   22$ \\ \hline
$a(n)$    & $   67$ & $  100$ & $  218$ & $  264$ & $  298$ & $  328$ & $  368$ & $  430$ & $  463$ & $  591$ & $  651$ \\ \hline
$p_{a(n)}$& $  331$ & $  541$ & $ 1361$ & $ 1693$ & $ 1973$ & $ 2203$ & $ 2503$ & $ 2999$ & $ 3299$ & $ 4327$ & $ 4861$ \\ \hline
\hline
$n$       & $   23$ & $   24$ & $   25$ & $   26$ & $   27$ & $   28$ & $   29$ & $   30$ & $   31$ & $   32$ & $   33$ \\ \hline
$a(n)$    & $  739$ & $  758$ & $  782$ & $  843$ & $  891$ & $  929$ & $ 1060$ & $ 1184$ & $ 1230$ & $ 1316$ & $ 1410$ \\ \hline
$p_{a(n)}$& $ 5623$ & $ 5779$ & $ 5981$ & $ 6521$ & $ 6947$ & $ 7283$ & $ 8501$ & $ 9587$ & $10007$ & $10831$ & $11777$ \\ \hline
\hline
$n$       & $   34$ & $   35$ & $   36$ & $   37$ & $   38$ & $   39$ & $   40$ & $   41$ & $   42$ & $   43$ & $   44$ \\ \hline
$a(n)$    & $ 1832$ & $ 2226$ & $ 3386$ & $ 3645$ & $ 3794$ & $ 3796$ & $ 4523$ & $ 4613$ & $ 4755$ & $ 5009$ & $ 5950$ \\ \hline
$p_{a(n)}$& $15727$ & $19661$ & $31469$ & $34123$ & $35671$ & $35729$ & $43391$ & $44351$ & $45943$ & $48731$ & $58889$ \\ \hline
\end{tabular}
\end{center}
\vspace{1mm}
If $1 \leq n \leq 43$ and
\begin{displaymath}
k \in \left[ \frac{p_{a(n+1)}}{p_{a(n+1)-1}}, \frac{p_{a(n)}}{p_{a(n)-1}} \right),
\end{displaymath}
then $R_1^{(k)} = p_{a(n)}$.
\end{enumerate}
\end{kor}

% 
% \vspace{5mm}
% 
% \textsc{Mathematisches Institut, Heinrich-Heine-Universität Düsseldorf, 40225 Düsseldorf}, \textsc{Germany}
% 
% \emph{E-mail address}: \texttt{axler@math.uni-duesseldorf.de}
% 
% \emph{E-mail address}: \texttt{lessmann@math.uni-duesseldorf.de}

\end{document}